\documentclass{au}

\usepackage{amsmath,amssymb,url}

\usepackage{mathrsfs}       

\usepackage{enumerate}  

\numberwithin{equation}{section}

\theoremstyle{plain}
\newtheorem{theorem}{Theorem}[section]
\newtheorem{lemma}[theorem]{Lemma}

\newtheorem{corollary}[theorem]{Corollary}

\theoremstyle{definition}

 \DeclareMathOperator{\image}{Im}

\newcommand{\bb}[1]{\mathbb{#1}}
   \newcommand{\A}{\bb{A}}
   \newcommand{\B}{\bb{B}}
   \newcommand{\LL}{\bb{L}}  
   \newcommand{\M}{\bb{M}}    
   \newcommand{\C}{\bb{C}} 
   
\newcommand{\cS}{\mathcal{S}}

\newcommand{\CC}{\mathscr{C}}
\newcommand{\sL}{\mathscr{L}}
\newcommand{\V}{\mathscr{V}}

\newcommand{\al}{\alpha}
\newcommand{\ka}{\kappa}
\newcommand{\lam}{\lambda}
\newcommand{\om}{\omega}
\newcommand{\Om}{\Omega}
\newcommand{\ph}{\varphi}
\newcommand{\sg}{\sigma}     
\newcommand{\thet}{\theta}     

\newcommand{\br}[1]{[\![#1]\!]}  

\newcommand{\epi}{\twoheadrightarrow}

\renewcommand{\ge}{\geqslant}       
\newcommand{\join}{\bigvee}
\renewcommand{\le}{\leqslant}  
\newcommand{\meet}{\bigwedge}
\newcommand{\mono}{\rightarrowtail}
\newcommand{\ov}[1]{\overline{#1}}
        
\newcommand{\sub}{\subseteq}

\renewcommand{\>}{\rangle}

\begin{document}


\title[Canonical extensions and ultraproducts of polarities]{Canonical extensions and ultraproducts of polarities}

\author[Robert Goldblatt]{Robert Goldblatt}
\email{rob.goldblatt@msor.vuw.ac.nz}
\urladdr{sms.vuw.ac.nz/~rob}
\address{School of Mathematics and Statistics\\
Victoria University of Wellington\\New Zealand}

\thanks{The author thanks Mai Gehrke and Ian Hodkinson for some very helpful comments, information and improvements.}

\dedicatory{In memoriam Bjarni J{\'o}nsson}



\begin{abstract}
J{\'o}nsson and Tarski's notion of the  perfect extension of a Boolean algebra with operators has evolved into an extensive theory of canonical extensions of lattice-based algebras. After reviewing this evolution we make  two contributions. First it is shown that the failure of  a variety of algebras to be closed under canonical extensions is witnessed by a particular one of its free algebras.
The size of the set of generators of this algebra can be made a function of a collection of varieties and  is a kind of Hanf number for canonical closure. Secondly we study the complete lattice of stable subsets of a polarity structure, and show that if a class of polarities is closed under ultraproducts, then its stable set lattices generate a variety that is closed under canonical extensions. This generalises an earlier result of the author about generation of canonically closed varieties of Boolean algebras with operators, which was in turn an abstraction of the result that a first-order definable class of Kripke frames determines a modal logic that is valid in its so-called canonical frames.
\end{abstract}

\maketitle


\section{A biography of canonical extension}\label{sec1}

In  a 2007 conference abstract \cite{jons:cano07}, Bjarni J{\'o}nsson described `an acorn from which a mighty oak has grown'. He was referring to a theorem communicated to him in 1946 by Tarski, stating that every relation algebra can be extended to a complete and atomic relation algebra. It was part of Tarski's project to modernise the Nineteenth Century theory of relations. But J{\'o}nsson realised that Tarski's construction applied to  other kinds of algebra, and this led to their celebrated two-part work \cite{jons:bool51,jons:bool52} on Boolean algebras with operators (BAO's), these operators being additional finite-join preserving maps.  Paper \cite{jons:bool51} presented the general theory of BAO's with application to closure, cylindric and projective algebras; while \cite{jons:bool52} was entirely devoted to relation algebras.

The Extension Theorem of \cite{jons:bool51} showed that  each BAO  $\B$ has an extension $\B^\sg$, which they called the `perfect' extension of $\B$, and which is complete and atomic with its additional operators being completely join preserving. Moreover, any BAO meeting the latter description is isomorphic to the `complex algebra' of all subsets of some relational structure, with the $n$-ary operators of the complex algebra  defined out of  the $n+1$-ary relations of the structure. From this followed a Representation Theorem: any BAO $\B$ has an associated relational structure $\B_+$ such that $\B$ is embeddable into the complex algebra of all subsets of $\B_+$. These results were first announced in \cite{jons:bool48}.

The notion of perfect extension was built on Stone's representation of a Boolean algebra as a field of sets \cite{ston:repr36}. But J{\'o}nsson and Tarski took a more refined approach to this, defining $\B^\sg$ abstractly to be any complete and atomic extension of $\B$ that satisfies certain lattice-theoretic axioms of separation and compactness, and proving  the  \emph{uniqueness} of such an extension up to a unique isomorphism over $\B$. Stone's theory then facilitated the proof of \emph{existence} of $\B^\sg$ by providing a perfect extension of the Boolean reduct of $\B$.
A significant part of the analysis of \cite{jons:bool51} was devoted to showing how each operator $f$ on $\B$ has a \emph{canonical extension} to a completely join preserving operation $f^\sg$ on $\B^\sg$.

Perfect extensions were used in the theory of cylindric algebras \cite{henk:cyli71}, where they were called \emph{canonical embedding algebras}. That  name was later used by the present author in some papers \cite{gold:vari89,gold:clos91,gold:elem95}, about dualities between BAO's and relational structures,  which were inspired by interim developments  in the study of modal logics. The algebraic models of a modal logic form a variety of BAO's that have a single unary operator interpreting the possibility modality. $\B_+$ was dubbed the \emph{canonical structure of} $\B$ in \cite{gold:vari89}, because when $\B$ is a Lindenbaum algebra of a modal logic (i.e.\ a freely  generated algebraic model), then $\B_+$ is isomorphic to a \emph{canonical Kripke frame} of the logic, a kind of structure  that had been developed by  modal logicians 
\cite{baya:quas59,cres:henk67,lemm:inte66,maki:comp66} by taking  points of a frame to be maximally consistent sets of formulas in the sense of Henkin \cite{henk:comp49}. The naming of such structures as `canonical' is due to Segerberg \cite{sege:deci68,sege:essa71}. Many logics were shown to be characterised by Kripke frames satisfying particular conditions, by showing that their canonical frames $\B_+$ satisfy those conditions and hence validate the logic. Logics for which this holds were called \emph{canonical}. The property implies that $\B^\sg$ is an algebraic model for the logic and therefore belongs to the associated variety.  Consequently, a class of BAO's was called \emph{canonical} in \cite{gold:vari89} when it is closed under canonical embedding algebras.

The question of which varieties are canonical is the question of which sets of equations are preserved by the $\B^\sg$ construction. This was addressed for individual equations already in \cite{jons:bool51}, where it was shown that all equations between strictly positive terms (ones not involving Boolean complements) are preserved. Also it was demonstrated that a number of such equational properties of a complex algebra corresponded to elementary properties of its underlying relational structure.  For example it was deduced from these results that if $\B$ is a closure algebra in the sense of McKinsey and Tarski \cite{mcki:alge44}, then $\B_+$ is a preordered set. This implies that the canonical frames of the modal logic S4 are  preordered. That yields a completeness theorem for S4 with respect to validity in preordered frames. Other results from \cite{jons:bool51} imply a corresponding analysis for other well-known logics, including S5, T and B.
But the investigation of canonicity for modal logics proceeded independently of the work of J{\'o}nsson and Tarski once the Kripke semantics emerged fifteen years or so later. Sahlqvist \cite{sahl:comp75} gave a syntactic definition of a wide class of modal formulas each of which has a first-order definable class of Kripke frames that includes the canonical frames of the logic axiomatised by that formula.
de Riijke and Venema \cite{deri:sahl95} generalised this formalism algebraically, defining  \emph{Sahlqvist equations} for BAO's of any  type and showing that they specify varieties that are canonical.

Note that if a variety is canonical then it is axiomatised by a set $E$ of equations which are preserved by canonical extensions when all taken together, but it does not follow that each member of $E$ is preserved by canonical extensions on its own. In fact there exist canonical varieties that are only \emph{barely canonical} in the sense that any axiomatisation of them must involve infinitely many axioms that are individually not canonical. The first examples of varieties with this property  were given by Hodkinson and Venema \cite{hodk:cano05}, and include the variety \textbf{RRA} of representable relation algebras. Many more can be found in  \cite{gold:mcki07,buli:bare13,kiko:dich15}. 

Four decades after the initial work on BAO's, J{\'o}nsson returned to the subject and began a productive collaboration with Mai Gehrke, who was a postdoc at Vanderbilt University during 1988--1990.
The $\B^\sg$ notation was introduced for what was now called the \emph{canonical extension} of $\B$. An elegant algebraic demonstration of the canonicity of Sahlqvist equations was given and  the functoriality of the action $f\mapsto f^\sg$ of canonical extension on various maps $f$ was explored \cite{jons:surv93,jons:cano94,jons:pres95}. The theory was lifted in \cite{gehr:boun94} from Boolean algebras to bounded distributive lattices with operators (DLO's), replacing Stone duality by Priestley duality \cite{prie:repr70} and atoms by completely join-irreducible elements, and showing that every variety of DLO's  is canonical. A sequel paper
\cite{gehr:mono00} enlarged the class of algebras to bounded distributive lattices expanded by additional operations that are monotone (isotone or antitone) in each coordinate. Canonical extension was shown to define a functor on the category of homomorphisms between these monotone bounded distributive lattice expansions (DLM's) that preserves and reflects injections and surjections. A corollary is that if a class of DLM's is canonical and closed under direct products, then the variety it generates is canonical. 

A significant innovation in \cite{gehr:mono00} was the use of tthe Boolean products of Burris and Werner \cite{burr:shea79}, and a demonstration that 
\begin{enumerate}
\item
if $\bb D$ is a Boolean product of a family $\{\bb D_i : i\in I\}$ of DLM's, then the canonical extension $\bb D^\sg$  is isomorphic to the direct product $\prod_I\bb D_i^\sg$ of the canonical extensions of its factors.
\end{enumerate}
 That result leads to a proof that 

\begin{enumerate}
\item[(2)]  
if a class of DLM's is canonical and closed under \emph{ultra}products, then the variety it generates is canonical. 
\end{enumerate}
This  was viewed as an `algebraic counterpart' to a theorem of the present author from \cite{gold:vari89}, which was itself abstracted from a theorem about canonicity of modal logics due to Fine \cite[Theorem 3]{fine:conn75}. Fine's theorem states that if a class of Kripke frames is definable in first-order logic, then the modal logic that it determines is canonical.
In \cite[Theorem 3.6.7]{gold:vari89} this result was generalised to the statement that
\begin{enumerate}
\item[(3)]    \label{myFine}
if a class $\cS$ of relational structures (of any given type) is closed under ultraproducts, then the complex algebras of the members of $\cS$  generate a variety of BAO's that is canonical.    
\end{enumerate}
A second proof of this result was given in \cite{gold:clos91}, and some strengthenings of it in \cite{gold:elem95} (see also \cite{gold:fine16} for a review of this work).
One objective of the present paper is to generalise (3) further to varieties of lattices in place of BAO's.

Now in the Gehrke--J{\'o}nsson papers \cite{gehr:boun94, gehr:mono00}, the definition of the canonical extension $f^\sg$ of a map was essentially the same as in the J{\'o}nsson--Tarski original \cite{jons:bool51}. It has the limitation that $f^\sg$ is only guaranteed to be an extension of $f$ when $f$ is  isotone. Gehrke devised a new topologically motivated definition of $f^\sg$ that extends an arbitrary $f$ and agrees with the old definition for isotone $f$. It was used in a third joint paper with J{\'o}nsson on distributive lattices \cite{gehr:boun04}, now focused on bounded distributive lattice expansions (DLE's), i.e.\ arbitrary algebras based on bounded distributive lattices. Result (1) above was shown to hold for DLE's, and  was used to prove that the canonical extension
$(\prod_I{\bb D_i})^\sg $ of a direct product of DLE's is isomorphic to the direct product
$\prod_{U\in \beta I} (\prod_U{\bb D_i})^\sg $
of the canonical extensions of  all ultraproducts  $\prod_U{\bb D_i}$ of the algebras $\bb D_i$ by ultrafilters $U$ (here $\beta I$ is the set of all ultrafilters on $I$). That led to a version of result (2) above for DLE's, and then to  another proof of result (3) for BAO's.  As well as providing significant information about canonical extensions of direct products, this new proof of (3) works more on the algebraic side of the duality between BAO's and relational structures, and despite being about relational structures in an essential way, does not require any knowledge of what category they form, i.e.\ what are the morphisms of relational structures that are dual to BAO-homomorphisms. We will make use of this proof strategy below in generalising result (3).

The new definition of $f^\sg$ was used by
Gehrke and Harding \cite{gehr:boun01} to introduce a notion of canonical extension for bounded lattice expansions (LE's), i.e.\   bounded lattices with arbitrary additional operations. In an earlier paper \cite{hard:cano98}, Harding had constructed canonical extensions of lattices by using an embedding of a bounded lattice into the lattice of stable elements of a Galois connection on a Boolean algebra, and invoking the canonical extension of the Boolean algebra. In \cite{gehr:boun01}  a more direct approach was taken, giving  a new axiomatic definition of $\LL^\sg$ for a bounded lattice $\LL$, obtained by replacing the separation property by a density condition that we will describe later. 
Canonical extension was shown to yield a functor on the category of monotone LE's that preserves and reflects monomorphisms and epimorphisms.  Result (1) was proven for LE's, and result (2) for monotone LE's. The existence of $\LL^\sg$ was established by constructing it as the lattice of stable subsets of a Galois connection induced by a polarity between filters and ideals of $\LL$. Other constructions are possible, based on representations of lattices by Urquhart \cite{urqu:topo78}, Hartung \cite{hart:topo92} and others. The relationships between several such incarnations of $\LL^\sg$ have been worked out in \cite{crai:cano12,crai:fres13,crai:reco14}.

The relationship between canonical extensions and completions in the sense of MacNeille \cite{macn:part37} was clarified by Gehrke, Harding and Venema \cite{gehr:macn06} using a  novel and elegant proof method, based on ideas from nonstandard analysis, to show that if $\LL$ is a monotone LE, then $\LL^\sg$ has a complete lattice embedding into the MacNeille completion of any sufficiently saturated elementary extension of $\LL$. Hence any equation, and indeed any universal sentence, about monotone LE's that is preserved by MacNeille completions must be preserved by canonical extensions. We will be applying this embedding from \cite{gehr:macn06} in generalising result (3) below.

Having generalised canonical extensions by dispensing with  Boolean complements and then distributivity, it remained to dispense with  the lattice structure itself. Canonical extensions of posets and monotonic poset expansions were defined by Dunn, Gehrke and Palmigiano \cite{dunn:cano05} and studied further in \cite{gehr:cano08,gehr:delt13,suzu:cano11,mort:cano14}. They have been applied to develop relational models of substructural and other kinds of logic, including linear logic, relevant logic, and the Lambek calculus \cite{gehr:gene06,alme:cano09,gehr:cano10,suzu:cano11a,cher:gene12,coum:rela14}.

Over the last dozen years  there have been many other topological, algebraic, categorical and logical studies of, or involving, canonical extensions 
\cite{bezh:prof06,bezh:topo08,coum:gene12,dave:bool07,dave:topo11,dave:cano12,gehr:cano14,gehr:dist14,gehr:sahl05,gehr:cano07,gehr:dual07,gehr:view09,gehr:cano11,gonz:topo16,gool:dua12,gouv:prof13,gouv:cano14,hard:prof06,havi:cano06,mosh:topo14,mosh:topo14a,vosm:logi10}.
The acorn has become a forest.

The present paper makes two contributions to the theory of canonical extensions of lattice-ordered algebras. 
Here is a summary of its contents. Sections 2 and 3 review definitions and results about canonical extensions that we will use.
Section 4 contains the first contribution. For any non-canonical variety $\V$ of lattice-based algebras it defines $\ka_\V$ to be the least cardinal $\ka$ such that the canonical extension $\LL_\ka(\V)^\sg$ is not in $\V$, where $\LL_\ka(\V)$ is the free algebra in $\V$ on $\ka$-many generators. Examples are given of varieties of modal algebras for which $\ka_\V$
is 0, 1 or $\om$. Known results about the varieties of modular ortholattices and of orthomodular lattices imply that they have  $\ka_\V=3$.
 It is also shown that for any collection $\Om$ of varieties of a given signature there exists a cardinal $\ka_\Om$ such that for any variety $\V$ from $\Om$, if the canonical extension of $\LL_{\ka_\Om}(\V)$ belongs to  $\V$, then so does the canonical extension of every other member of $\V$. Moreover $\ka_\Om$ is the least cardinal with this property and can be thought of as  an analogue for canonical closure of the  notion of the Hanf number of a formal logic. We  show further that, assuming the operation $\LL\mapsto\LL^\sg$ preserves certain homomorphisms,  the role of 
$\LL_{\ka_\Om}$ can also be fulfilled by the $\ka_\Om$-th direct power
$\LL_\om\!^{\ka_\Om}$ of the free algebra on denumerably many generators (see Theorem \ref{equivalences}).

Sections 5--7 contain the second contribution, which is an adaptation of the result (3) above about generating a canonical variety from an ultraproducts-closed class of relational structures. We work with the notion of a \emph{polarity} as a structure $P=(X,Y,R)$ comprising a binary relation $R\sub X\times Y$. This relation induces a Galois connection between the powersets of $X$ and $Y$, leading to a notion of a \emph{stable} subset of $X$.  The set $P^+$ of stable subsets is a complete lattice. (An instance of this construction was used in \cite{gehr:boun01} to define a canonical extension of any bounded lattice, as already mentioned.) In Section 5 we define the ultraproduct $\prod_U P_i$ of a family of polarities $P_i$ and show that the ultraproduct $\prod_U(P_i^+)$ of the stable set lattices $P_i^+$ has an embedding into the lattice 
$( \prod_UP_i)^+ $ of the ultraproduct of the $P_i$'s.
The ultra\emph{power} case of this yields  an embedding  of any ultrapower $(P^+)^U$ of a stable set lattice into  the lattice $(P^U)^+$. 
In Section 6  we show that this embedding is a MacNeille completion of  $(P^+)^U$. Combining this with the result from \cite{gehr:macn06} on the embedding of canonical extensions into MaNeille completions, we obtain the conclusion that for any polarity $P$ there exists an ultrafilter $U$ such that the canonical extension $(P^+)^\sg$ of the stable set lattice of $P$ embeds into the stable set lattice $(P^U)^+$ of the  ultrapower  $P^U$.

In  Section 7 these results are combined with further analysis to show that if $\cS$ is any class of polarities that is closed under ultraproducts, then the variety of lattices generated by $\cS^+=\{P^+:P\in\cS\}$ is closed under canonical extensions.  An axiomatisation of the principles required for the proof is given, so that it can be applied to other situations. We also explain why the variety generated by $\cS^+$ is equal to the variety generated by
$\cS_\mathrm{el}^{\ +}$, where $\cS_\mathrm{el}$ is the smallest elementary class containing $\cS$.

\section{Canonical extensions}

We take all lattices to be bounded, and use the signature $\land,\lor,0,1$ to describe them.  We use $\le$ for the partial order of a lattice, and the symbols $\join$ and $\meet$ for the join and meet of a set of elements, when these exist. All lattice homomorphisms are assumed to preserve 0 and 1. A surjective homomorphism may be  called an \emph{epimorphism} and an injective homomorphism a \emph{monomorphism}. We often use the standard symbols
$\epi$ and  $\mono$ for epimorphisms and monomorphisms respectively.

A function $f\colon\LL\to\M$ between lattices is called \emph{isotone} if $a\le b$ implies $fa\le fb$, and  \emph{antitone} if $a\le b$ implies $fb\le fa$. It is a \emph{lattice embedding}  if it is a monomorphism of bounded lattices. A lattice embedding is always an \emph{order embedding}, i.e.\ has 
$a\le b$ iff $fa\le fb$. The notation $f[S]$ will be used for the image $\{fa:a\in S\}$ of a set $S$ under function $f$.

A \emph{completion} of lattice $\LL$ is a pair $(e,\C)$ with $\C$ a complete lattice and $e\colon\LL\mono\C$  a lattice embedding.
An element of $\C$ is called \emph{open} if it is a join of elements from the image $e[\LL]$ of $\LL$ and \emph{closed} if it is a meet of elements from $e[\LL]$. Note that members of $e[\LL]$ are both open and closed. The set of open elements of the completion is denoted $O(\C)$, or just $O$ if the embedding is understood. The  set of closed elements is denoted $K(\C)$, or just $K$.

A completion $(e,\C)$ of $\LL$ is \emph{dense} if $K(\C)$ is join-dense and $O(\C)$ is meet-dense in $\C$, i.e.\ if  every member of $\C$ is both a join of closed elements and a meet of open elements. Note that if $\LL$ is finite, then each of the join-density or the meet-density here is enough to make $e[\LL]=\C$ and so make $\LL$  isomorphic to $\C$ under $e$.

A completion is \emph{compact} if for any set $S$ of closed elements and any set $T$ of open elements such that $\meet S\le\join T$,  there are finite sets $S'\sub S$ and $T'\sub T$ with $\meet S'\le\join T'$. An equivalent formulation of this condition that we will be using is that for any subsets $S$ and $T$ of $\LL$ such that $\meet e[S]\le\join e[T]$ there are finite sets $S'\sub S$ and $T'\sub T$ with $\meet S'\le\join T'$.


A \emph{canonical extension} of bounded lattice $\LL$ is a completion $(e,\LL^\sg)$  of $\LL$ that is dense and compact. It is shown in \cite{gehr:boun01} that a dense and compact completion exists for any $\LL$, and that any two such completions are isomorphic by a unique isomorphism commuting with the embeddings of $\LL$. 

It is convenient to assume that $\LL$ is a sublattice of $\LL^\sg$, with the embedding $e\colon\LL\to\LL^\sg$ being the inclusion function $e(a)=a$.
In particular this assumption is made in taking an arbitrary map $f\colon\LL\to \M$ between lattices, viewing it as a map $f\colon\LL\to \M^\sg$, and then lifting it to a map $f^\sg\colon\LL^\sg\to\M^\sg$, the \emph{canonical extension of $f$}, by using the lattice completeness of $\M^\sg$ to define, for all $x\in\LL^\sg$,
\begin{align*}
f^\sg x &=\join\big\{\meet \{fa:a\in\LL\text{ and }p\le a\le q\}:K(\LL^\sg)\ni p\le x\le q\in O(\LL^\sg)\big\}.
\end{align*}
Then $f^\sg$ extends $f$. If $f$ is isotone,  $f^\sg$ is also isotone and has the simpler description
\begin{equation}  \label{canextf}
f^\sg x =\join\big\{\meet \{fa:p\le a\in\LL\}:x\ge p\in K(\LL^\sg)\}.
\end{equation}

\begin{lemma}\cite{gehr:boun01}  \label{presfsg}
If $f:\LL\to\M$ is a bounded lattice homomorphism, then $f^\sg$ is a complete lattice homomorphism, i.e.\ $f^\sg$ preserves all joins and meets. Moreover, $f^\sg$ is injective or surjective if, and only if, $f$ has the same property.
\qed
\end{lemma}
Now if $f\colon\LL^n\to\LL$ is an $n$-ary operation on $\LL$, then $f^\sg$ is a map from $(\LL^n)^\sg$ to $\LL^\sg$. But $(\LL^n)^\sg$ can be identified with $(\LL^\sg)^n$, since the inclusion embedding $\LL^n\to(\LL^\sg)^n$ is dense and compact, so this allows $f^\sg$ to be regarded as an $n$-ary operation on $\LL^\sg$.

A \emph{lattice-based algebra} is an algebra $\LL=(\LL_0,\{f_j:j\in J\})$ comprising a bounded lattice $\LL_0$ with additional finitary operations $f_j$ on $\LL_0$. Then the canonical extension of $\LL$ is defined to be the algebra  $\LL^\sg=(\LL_0^\sg,\{f_j^\sg:j\in J\})$ of the same similarity type. 
A class of lattice-based algebras of a given type is called \emph{canonical} if is closed under canonical extensions, i.e.\ if it  contains $\LL^\sg$ whenever it contains $\LL$.

It has been shown by Vosmaer \cite[3.3.12]{vosm:logi10} that if $h\colon\LL\epi\M$ is a \emph{surjective} homomorphism between lattice-based algebras of the same type, then the surjective $h^\sg\colon\LL^\sg\to\M^\sg$ is also a homomorphism for that type of algebra. We apply this fact to the study of the role of free algebras in canonical closure.
Standard theory tells us that in order for  for a class  of algebras of some type that contains a constant to contain a free algebra on any set of generators,  it suffices that it be closed under isomorphism, subalgebras and direct products, and be nontrivial i.e.\ have a member containing more than one element (e.g.\ \cite[Theorem 4.117]{mcke:alge87}). Since we are dealing with lattices having 0 and 1, any nontrivial variety of  lattice-based algebras possesses all free algebras.

\begin{lemma} \label{canonicalfree}
A variety $\V$ of lattice-based algebras is canonical if, and only if, it contains the canonical extensions of all of its  free algebras. 
\end{lemma}
\begin{proof}
If $\V$ is canonical, it is immediate that it contains the canonical extension of any of its free algebras. For the converse, if 
$\V$ is not canonical then there exists an $\M\in\V$ with $\M^\sg\notin \V$. Thus $\M\not\cong\M^\sg$, so $\M$ must be infinite. Hence $\V$ is a non-trivial variety, so it has all free algebras. In particular there is a  free algebra $\LL$ in $\V$ with an epimorphism $h\colon\LL\epi\M$. Then $h^\sg$ is an epimorphism from $\LL^\sg$ onto $\M^\sg$, as  noted above.  Since $\M^\sg\notin \V$, it follows that $\LL^\sg\notin \V$ as $ \V$ is closed under homomorphic images.
\end{proof}
Note that the free algebra $\LL$ in this proof can be assumed to have infinitely many generators. So for a variety to be canonical it is sufficient that it contains the canonical extensions of all of its infinitely generated free algebras.

\section{Products of ultraproducts}

We  review the definition of ultraproducts. If $f$ and $g$ are functions with the same domain $I$, we denote by $\br{f=g}$ their equaliser set $\{i\in I:f(i)=g(i)\}$.
Given a set $\{\A_i:i\in I\}$ of algebras of the same type, and an ultrafilter $U$ on $I$, define a relation $\sim_U$ on the product algebra $\prod_I\A_i$ by putting $f\sim_U g$ iff $\br{f=g}\in U$.  Then $\sim_U$ is a congruence relation on $\prod_I\A_i$, and the  resulting quotient algebra of $\prod_I\A_i$ is called the \emph{ultraproduct} of  $\{\A_i:i\in I\}$ with respect to $U$, denoted $\prod_U\A_i$. The elements of the ultraproduct are the equivalence classes $f^U=\{g\in\prod_I A_i:f\sim_U g\}$ with $f\in\prod_IA_i$. The map $f\mapsto f^U$ is an epimorphism from $\prod_I\A_i$ to $\prod_U\A_i$.
When all the factors $\A_i$ are equal to a single algebra $\A$, then the ultraproduct is called an \emph{ultrapower} of $\A$, denoted $\A^U$.

An algebra $\A$ is a \emph{Boolean product} of the algebras $\{\A_i:i\in I\}$ if
 $\A$ is a subdirect product of the $\A_i$'s and 
there is a Boolean space topology on $I$ such that $\br{f=g}$ is clopen for all $f,g\in \A$; and for all such $f$ and $g$ and any clopen set $N\sub I$, the function $f\restriction_N{\cup}\, g\restriction_{I- N}$ belongs to $\A$.

Given a set of algebras  $\{\A_i:i\in I\}$, let $\beta I$ be the set of all ultrafilters on $I$, with the Stone space topology. Then the product algebra $\prod_I\A_i$ can be represented as a Boolean product of the family $\{\prod_U\A_i:U\in \beta I\}$ of ultraproducts. The product of the projections  
$\prod_I\A_i\twoheadrightarrow \prod_U\A_i$, i.e.\ the homomorphism
\begin{equation}\textstyle     \label{bprodrep}
\thet\colon\prod_I\A_i  \longrightarrow  \underset{U\in \beta I}{\prod} (\prod_U\A_i)
\end{equation}
 defined by $\thet(f)(U)=f^U$, maps $\prod_I\A_i$ isomorphically onto a Boolean product of the $\prod_U\A_i$'s. This is proved in \cite[Theorem 2.1]{burr:shea79} (which actually gives the more general result that any reduced product of the $\A_i$'s has such a Boolean product representation).

Now it was shown by Gehrke and Harding \cite[Lemma 6.7]{gehr:boun01} that if a lattice-based algebra $\LL$ is a Boolean product of a collection $\{\LL_i:i\in I\}$, then its canonical extension $\LL^\sg$ is isomorphic to the product $\prod_I\LL_i^\sg$ of the canonical extensions of the factors $\LL_i$. Combining this with the  Boolean product representation of a  direct product produces the following result, given by Gehrke and J\'onsson \cite[3.19]{gehr:boun04} for distributive lattice expansions.

\begin{theorem}  \label{boolrepsig}
For any set $\{\LL_i:i\in I\}$ of lattice-based algebras of the same type, there is an isomorphism
\begin{equation}\textstyle
\big(\prod_I\LL_i\big)^\sg \cong \underset{U\in \beta I}{\prod}\big(\prod_U\LL_i\big)^\sg
\end{equation}
between the canonical extension  of the product of the $\LL_i$'s and the product
of the canonical extensions of all the ultraproducts of the $\LL_i$'s.
\end{theorem}
\begin{proof}
The result follows by putting $\A_i=\LL_i$ in  \eqref{bprodrep} and then applying \cite[6.7]{gehr:boun01}, as just noted. Here we work through a direct proof using the properties of ultrafilters on $I$.

Put $\LL= \prod_I\LL_i$ and let $\thet\colon\LL\to\prod_{\beta I}(\prod_U\LL_i)$ have $\theta(f)=\<f^U:U\in\beta I\>$
as in  \eqref{bprodrep}. Since each algebra 
$(\prod_U\LL_i)^\sg$ is an extension of $\prod_U\LL_i$, we may view $\thet$ as a homomorphism from $\LL$ to 
$\prod_{\beta I}(\prod_U\LL_i)^\sg$. As such we prove that $\thet$ makes $\prod_{\beta I}(\prod_U\LL_i)^\sg$ a canonical extension of $\LL$, from which the Theorem follows by the uniqueness of canonical extensions up to isomorphism.

First note that $\prod_{\beta I}(\prod_U\LL_i)^\sg$ is a product of complete lattices, hence is complete. Indeed the meet and join of any subset $S$ of the product is determined pointwise by meets and joins in each factor:
\begin{equation}  \label{prodmj}
\textstyle
(\meet S)(U)= \meet\{h(U):h\in S\} \text{ and }
(\join S)(U)= \join\{h(U):h\in S\}. 
\end{equation}  
To show that $\theta$ is injective, let $f$ and $g$ be distinct elements of $\LL$. Then their equaliser $\br{f=g}$ is not equal to $I$, so fails to belong to some ultrafilter $U\in\beta I$. Then $f^U\ne g^U$ in $\prod_U\LL_i$, which is enough to ensure that $\thet(f)\ne\thet(g)$ as required for injectivity. Hence $\theta$ is a lattice embedding.

To show that $\theta$ is dense we need to show that any member $h$ of $\prod_{\beta I}(\prod_U\LL_i)^\sg$ is a meet of joins of subsets of $\thet[\LL]$. Given any  $U\in\beta I$ and any $a\in (\prod_U\LL_i)^\sg$, define the function $\widehat{a}$
on $\beta I$ by putting $\widehat{a}(U)=a$ and $\widehat{a}(U')=1$ in $(\prod_{U'}\LL_i)^\sg$ if $U'\ne U$.
Then $h$ is the meet in $\prod_{\beta I}(\prod_U\LL_i)^\sg$ of $\{\widehat{h(U)}:U\in\beta I\}$. So our problem reduces to showing that any member of the form $\widehat{a}$ is a meet of joins of subsets of $\thet[\LL]$.

Now given $a\in (\prod_U\LL_i)^\sg$, by the density of the embedding of 
 $\prod_U\LL_i$ in  $(\prod_U\LL_i)^\sg$ we have that $a=\meet A$ for some $A\sub (\prod_U\LL_i)^\sg$ such that each $q\in A$ is the join of some subset $S_q$ of  $\prod_U\LL_i$. Then $\widehat{a}=\meet\{\widehat{q}:q\in A\}$ in 
$\prod_{\beta I}(\prod_U\LL_i)^\sg$. So the problem about $\widehat{a}$ can be solved by showing that if $q\in A$, then $\widehat{q}$ is the join in $\prod_{\beta I}(\prod_U\LL_i)^\sg$ of some subset of $\thet[\LL]$. 
Note that if $S_q=\emptyset$, so $q=\join\emptyset=0$, we can change $S_q$ to $\{0\}$ without changing $\join S_q$. Thus we can assume that $S_q$ is non-empty, from which it can be shown that
$\widehat{q}=\join\{\widehat{s}:s\in S_q\}$ in $\prod_{\beta I}(\prod_U\LL_i)^\sg$.
 For each $s\in S_q$ there is some $f_s\in\LL$ such that $s=f_s{}^U$. Then for each $J\in U$, define $f_{s,J}$ to be the member of $\LL$ that agrees with $f_s$ at each $i\in J$ and takes the value $1\in\LL_i$ for $i\in I- J$. Since $\br{f_{s,J}=f_s}\supseteq J\in U$, this gives $(f_{s,J})^U=f_s{}^U=s$. Moreover, if $U\ne U'\in\beta I$ then there is some $J\in U$ with 
 $I- J\in U'$, so   $(f_{s,J})^{U'}=1$ in $(\prod_{U'}\LL_i)^\sg$. This shows that $\thet(f_{s,J})=\widehat{s}$ for all $J\in U$. Now put
 $T_q=\{f_{s,J}: s\in S_q \text{ and } J\in U\}\sub\LL$. Then $\join\thet[T_q]=\join\{\widehat{s}:s\in S_q\}=\widehat{q}$, showing that $\widehat{q}$ is the join of a  subset of $\thet[\LL]$, as required to finish the proof that any 
$h\in\prod_{\beta I}(\prod_U\LL_i)^\sg$ is a meet of joins of subsets of $\thet[\LL]$. A dual argument shows that $h$ is also a join of meets of subsets of $\thet[\LL]$, so the embedding $\thet$ is dense.

It remains to show that  $\thet$ is a compact embedding. Take any subsets $S$ and $T$ of $\LL$, let $S^\land_\om$ be the set of all meets of finite subsets of $S$ in $\LL$, and let $T^\vee_\om$ be the set of all joins of finite subsets of $T$. Assume that $f\not\le g$  for all $f\in S^\land_\om$  and $g\in T^\vee_\om$. To prove compactness we must show that
 $\meet e[S]\not\le\join e[T]$.
 For any $f,g\in\LL$, let $\br{f\le g}=\{i\in I:f(i)\le g(i)\}$.
The partial order in any ultraproduct $\prod_U\LL_i$ has $f^U\le g^U$ iff $\br{f\le g}\in U$. Now define
\[
F=\{I- \br{f\le g}:f\in S^\land_\om \text{ and } g\in T^\vee_\om\}.
\]
We show that the collection $F$ of subsets of $I$ has the finite intersection property.
Given $f_1,\dots,f_n\in S^\land_\om$ and $g_1,\dots,g_n\in T^\vee_\om$, put
$f=f_1\land\cdots\land f_n$ and $g=g_1\lor\cdots\lor g_n$. Then $f\in S^\land_\om$ (as a finite meet of finite meets is a finite meet etc.)  and similarly $g\in T^\vee_\om$, so by assumption $f\not\le g$. Hence  $\theta(f)\not\le \theta(g)$ as $\thet$ is a lattice embedding. As the order in 
$\prod_{\beta I}(\prod_U\LL_i)$ is defined pointwise, it follows that there is some $U_0\in\beta I$ such that
 $\theta(f)(U_0)\not\le \theta(g)(U_0)$ in $\prod_{U_0}\LL_i$, i.e.\ $f^{U_0}\not\le g^{U_0}$ and so
 $\br{f\le g}\notin U_0$. Hence some $i_0\in I$ has $f(i_0)\not\le g(i_0)$.
 But for all $k\leq n$, $f\le f_k$ and $g_k\le g$, so this implies  $f_k(i_0)\not\le g_k(i_0)$.
Thus $i_0\in\bigcap_{k\leq n} (I-\br{f_k\le g_k})$.
That confirms that $F$ has the finite intersection property. Therefore there is some $U_1\in\beta I$ with $F\sub U_1$.

Now suppose, for the sake of contradiction, that
$\meet\{\thet(f):f\in S\}\le\join\{\thet(g):g\in T\}$.  Then by \eqref{prodmj},
$\meet\{f^{U_1}:f\in S\}\le\join\{g^{U_1}:g\in T\}$ in $(\prod_{U_1}\LL_i)^\sg$. But the latter is a compact extension of 
$\prod_{U_1}\LL_i$, so then 
\begin{equation}  \label{compred}
\meet\{f^{U_1}:f\in S'\}\le\join\{g^{U_1}:g\in T'\}
\end{equation}
for some finite sets $S'\sub S$ and $T'\sub T$.
Put $f_1=\meet S'\in S^\land_\om$ and $g_1=\join T'\in T^\lor_\om$. As the map $h\mapsto h^{U_1}$ on $\LL$ preserves finite meets and joins, \eqref{compred} asserts that $f_1{}^{U_1}\le g_1{}^{U_1}$. This implies that  $ \br{f_1\le g_1}\in U_1$, which is in contradiction with the fact that
$(I- \br{f_1\le g_1})\in U_1$ by construction. So we cannot have
 $\meet e[S]\le\join e[T]$ after all, and that completes the proof of compactness.
\end{proof}
\section{A Hanf number for canonicity}

We are going to focus  the failure of a variety to be canonical on its failure to contain the canonical extension of a particular one of its free algebras. This analysis is then extended to collections of varieties.

Let  $\V$ be  a class of  lattice-based algebras of some type such that there is a free algebra in $\V$ on any set of generators. 
(Recall the statement from just prior to Lemma \ref{canonicalfree} that this holds if $\V$ is any non-trivial variety.)
For each cardinal $\ka$, the free algebra  in $\V$ on $\ka$-many generators will be denoted  $\LL_\ka(\V)$, or just $\LL_\ka$ if $\V$  is understood. Recall the result of \cite[3.3.12]{vosm:logi10} that the canonical extension of an epimorphism is an epimorphism.

\begin{lemma} \label{kdashclosed}
If $\ka'\leq\ka$, and $\LL_\ka(\V)^\sg$ belongs to $\V$, then so does $\LL_{\ka'}(\V)^\sg$.
\end{lemma}

\begin{proof}
If $\ka'\leq\ka$, then by freeness of $\LL_\ka(\V)$ there is an epimorphism 
$\thet\colon$ $\LL_\ka(\V) \twoheadrightarrow \LL_{\ka'}(\V)$. Hence $\thet^\sg$ is an epimorphism making $\LL_{\ka'}(\V)^\sg$  a homomorphic image of 
$\LL_{\ka}(\V)^\sg$. But any variety is closed under homomorphic images.
\end{proof}

If  a variety $\V$ is not canonical,  then by Lemma \ref{canonicalfree} there is some 
$\ka$ such that the canonical extension $\LL_\ka(\V)^\sg$ of $\LL_\ka(\V)$ is not in $\V$. We define $\ka_\V$ to be the \emph{smallest}  such cardinal $\ka$.
This $\ka_\V$
may be thought of as the \emph{degree of canonicity} of $\V$. For by Lemma \ref{kdashclosed},
from $\LL_{\ka_\V}(\V)^\sg\notin\V$ it follows that $\LL_{\ka}(\V)^\sg\notin\V$ for all $\ka>\ka_\V$, 
so the class of cardinals $\{\ka: \LL_\ka(\V)^\sg\in\V\}$ is exactly the interval $\{\ka:0\leq\ka<\ka_\V\}$.
Thus the bigger $\ka_\V$ is, the more canonical closure a non-canonical $\V$ has.

There exist non-canonical $\V$ having $\ka_\V=0$. An example comes from  the consistent tense logic constructed  by S.~K.~Thomason \cite{thom:sema72}  that is not validated by any non-empty Kripke frame. Its variety $\V$ of algebraic models has $\LL^\sg\notin\V$ for every $\LL$ in $\V$ that has at least two elements. The free algebra $\LL_0(\V)$ is infinite and has $\LL_{0}(\V)^\sg\notin\V$. So in this case $\ka_\V=0$ while $\V$ is ``anti-canonical''.
 
 Another variety having $\ka_\V=0$ is the class of diagonalizable algebras. These are
modal algebras whose unary operator $f$ satisfies the equation $fa=f(a\land -fa)$.
They are the algebraic models of  the G\"odel-L\"ob modal logic of provability, also known as KW. The proof that this logic is not validated by its canonical Kripke frames (e.g.\  \cite[p.~51]{gold:logi92}) holds when the language is restricted to constant formulas with no variables, and when interpreted algebraically shows that if $\LL_0$ is the free diagonalizable algebra on 0 generators then $\LL_{0}{}^\sg$ is not a diagonalizable algebra. Alternatively one can work directly with the fact that $\LL_0$ is isomorphic to the complex algebra of finite or cofinite subsets of the structure $(\om,>)$  \cite[6.3]{smor:fixe82} to show that the meet in $\LL_{0}{}^\sg$ of the cofinite sets is an element that violates the given equation.
 
 For a case in which $\ka_\V=1$, let $\V$ be the variety of modal algebras validating the logic S4.Grz, also known as K1.1,
 which is the  logic characterised by the class of finite partially-ordered Kripke frames. The members of $\V$ are the closure algebras whose  unary operator satisfies Grzegorcyzk's axiom     \linebreak
 $a\le  f(a\land-f(-a\land fa))$.
 The free 0-generated algebra in this variety is the complex algebra of a one-element poset, i.e.\ the 2-element Boolean algebra with the identity function as closure operator.  So for this $\V$, $\LL_0$ is finite and therefore $\LL_0{}^\sg$ is isomorphic to $\LL_0$ and hence is in $\V$.
 But the proof in \cite{hugh:K1.182} that the logic is not validated by its canonical Kripke frames holds when the language is restricted to formulas with a single variable, and shows that $\LL_{1}(\V)^\sg\notin\V$. Hence $\ka_\V=1$.

 A notable case in which $\ka_\V=3$ is when $\V$ is the variety of modular ortholattices. Its free algebra $\LL_2$ on 2 generators is finite (with 96 elements \cite{kota:axio67}),  so $\LL_2{}^\sg\cong \LL_2\in\V$. But there exists a modular ortholattice $\LL$ with 3 generators that cannot be embedded into any complete modular ortholattice \cite{herr:fini81,brun:fini92}, so as $\LL^\sg$ is complete it cannot be a modular ortholattice (actually it is not modular, since it is an ortholattice by \cite{hard:cano98}). Since $\LL^\sg$ is an epimorphic image of $\LL_3{}^\sg$ it follows that $\LL_3{}^\sg$ is not in the variety either. Hence in this case $\ka_\V=3$.
 
A related example with  $\ka_\V=3$ is when $\V$ is the variety of orthomodular lattices. Here $\LL_2$ is finite and is the same algebra as the free 2-generated modular ortholattice of the previous example \cite[p.~229]{kalm:orth83}. By a construction in 
\cite[Prop.~3.4]{hard:cano98} there is a countable orthomodular lattice $\LL$ such that $\LL^\sg$ is not orthomodular. Since
$\LL^\sg$ is an epimorphic image of $\LL_\om{}^\sg$, it follows that $\LL_\om{}^\sg$ is not orthomodular. But by
\cite{hard:free02}, there is an embedding of $\LL_\om$ into the 3-generated orthomodular lattice $\LL_3$. From this it can be shown that  there is an embedding of $\LL_\om{}^\sg$ into $\LL_3{}^\sg$ (this depends on the fact that orthocomplementation is antitone: see the last paragraph of this Section). Hence $\LL_3{}^\sg$ is not orthomodular either.
 
Despite the individuality of all these examples, it seems plausible to conjecture that every natural number is equal to $\ka_\V$ for some variety $\V$.

 There are known varieties with $\ka_\V=\om$. In \cite{fine:logi74} Fine showed that the modal logic S4.3Grz is valid in all its finitely generated canonical frames, but not  in its $\om$-generated one. Interpreted algebraically, this implies that if $\V$ is the variety of all algebraic models of S4.3Grz, then $\LL_{\ka}(\V)^\sg\in \V$ for all $\ka<\om$ but  $\LL_{\om}(\V)^\sg\notin \V$, giving  $\ka_\V=\om$. Another example of this kind can be obtained from \cite[Section 6]{gold:elem95}, where a variety $\V$ of modal algebras is constructed that is not canonical but is locally finite, meaning that all of its finitely generated members are finite. In this case $\LL_{\ka}(\V)^\sg\cong \LL_{\ka}(\V)\in \V$ for all $\ka<\om$. The non-canonicity  involves an uncountable member $\B$ of $\V$ (a powerset). But the proof that $\B^\sg\notin \V$  depends on the properties of a particular countable subset of $\B$. If $\B_0\in\V$ is the subalgebra of $\B$ generated by this countable subset, then the proof shows that 
$\B_0{}^\sg\notin \V$. Since $\B_0{}^\sg$ is an epimorphic image of $\LL_{\om}(\V)^\sg$, it follows that $\LL_{\om}(\V)^\sg\notin \V$, so $\ka_\V=\om$ for this variety as well.
 
Nothing appears to be known about cases in which $\ka_\V$ is uncountable.  In particular, for varieties of modal algebras, the open question of whether there is one with $\ka_\V>\om$ dates back to \cite{fine:conn75}. 
 
 Now fix a signature for lattice-based algebras and a collection $\Om$ of nontrivial varieties of algebras of that signature.
Examples that might be considered are the collection of all nontrivial varieties of that signature; the collection of  varieties whose members are based on distributive lattices; the collection of  varieties of lattices with operators as additional operations, etc. 

The question of the legitimacy of forming $\Om$ as a collection of proper classes is dealt with in a standard way: the signature is a set (of symbols) whose associated equations form a set. So the class of all sets of equations for that signature  is a set. Since each variety is determined by the set of equations it satisfies, it follows that $\Om$ can be identified with a set (of sets of equations). In particular if we assign to each variety from $\Om$ a cardinal number, then the resulting collection of cardinals is a set.
Thus the collection of cardinals 
$$
S_\Om=\{\ka_\V: \V\in \Om \text{ and $\V$ is not canonical}\}
$$  
is a set,  and so has a supremum which we denote by $\ka_\Om$.   

\begin{theorem} \label{whencanonical}
\begin{enumerate}[\rm(1)]
\item    \label{whencanonical1}
A variety from $\Om$ is canonical if, and only if, it contains the canonical extension $\LL_{\ka_\Om}(\V)^\sg$ of its free algebra on $\ka_\Om$-many generators.
\item
For any variety $\V$ from $\Om$, if\/ $\LL_{\ka_\Om}(\V)^\sg$ is in $\V$ then  so is $\LL_{\ka}(\V)^\sg$ for all $\ka\geq\ka_\Om$. Moreover, $\ka_\Om$ is the least  cardinal with this property.
\end{enumerate}
\end{theorem}
\begin{proof}
(1):  If  $\V\in\Om$ is canonical, then it contains the canonical extensions of all its members, and in particular that of
$\LL_{\ka_\Om}(\V)$. For the converse, if $\V$ is not canonical, then $\ka_\V\in S_\Om$ and so $\ka_\V\leq\ka_\Om$.
But $\LL_{\ka_\V}(\V)^\sg\notin\V$ by definition of $\ka_\V$, hence by Lemma \ref{kdashclosed}, $\LL_{\ka_\Om}(\V)^\sg\notin\V$ as well.

(2): If $\LL_{\ka_\Om}(\V)^\sg\in\V$, then $\V$ is canonical by (1), so $\V$ contains $\LL_{\ka}(\V)^\sg$ for all $\ka$, hence  for all $\ka\geq\ka_\Om$.
Now suppose $\ka_0$ is a cardinal with the property that for all varieties $\V$ from $\Om$, if $ \LL_{\ka_0}(\V)^\sg\in\V$, then 
$\LL_{\ka}(\V)^\sg \in\V$ for all $\ka\geq\ka_0$. We show that $\ka_0$ is an upper bound of the set of cardinals $S_\Om$, giving $\ka_\Om\leq\ka_0$ as $\ka_\Om$ is the least upper bound of  $S_\Om$.

To see that $\ka_0$ is an upper bound,  take a variety $\V$ that is not canonical, so that  $\ka_\V$ is \emph{least} such that $\LL_{\ka_\V}(\V)^\sg\notin\V$. If we had $\ka_0<\ka_\V$, we would have $\LL_{\ka_0}(\V)^\sg\in\V$ by definition of $\ka_\V$, hence if $\V$ is from $\Om$ then $\LL_{\ka_\V}(\V)^\sg\in\V$ by supposition on $\ka_0$, a contradiction. Therefore  $\ka_\V \leq\ka_0$. So $\ka\leq\ka_0$ for all $\ka\in S_\Om$.
\end{proof}
Part (2) of this theorem suggests that $\ka_\Om$ is an analogue of the \emph{Hanf number} of a formal language (under a given model theory), defined as the least cardinal $\ka$ such that any set of sentences which has a model of size at least $\ka$ has arbitrarily large models (see \cite[Section 4.4]{bell:mode69}).  

The proof of part (2) can be adapted to show that $\ka_\Om$ is the least cardinal with the property given in (1), i.e.\ it is the least $\ka$ such that any variety $\V$ from $\Om$ is canonical if, and only if, it contains $\LL_{\ka}(\V)^\sg$.

If all varieties from $\Om$ are canonical, then $\ka_\Om=0$. By \cite{jons:bool51} this holds when $\Om$ is any collection of BAO varieties that are defined by strictly positive equations. 
Other examples mentioned in Section \ref{sec1}  for which it holds are collections of BAO varieties defined by Sahlqvist equations \cite{deri:sahl95} and collections of  varieties of bounded distributive lattices with operators  \cite{gehr:boun94}.

We will now show that in part (1) of  Theorem \ref{whencanonical},  $\LL_{\ka_\Om}$ can be replaced by  $\LL_{\om}$ when $\ka_\Om$ is finite, and by the direct power $\LL_\om\!^{\ka_\Om}$ under certain circumstances when
$\ka_\Om$ is infinite.

\begin{theorem}
Suppose $\ka_\Om$ is finite. Then a variety $\V$ from $\Om$ is canonical if, and only if, it contains  $\LL_{\om}(\V)^\sg$.
\end{theorem}
\begin{proof}
If $\V$ is not canonical, then $\LL_{\ka_\Om}(\V)^\sg$ is not in $\V$ and $\ka_\Om<\om$, so $\LL_{\om}(\V)^\sg$ is not in 
$\V$ by Lemma \ref{kdashclosed}.
\end{proof}

For the case of infinite $\ka_\Om$ we will use the following fact about varieties of algebras in general.

\begin{lemma}  \label{injhomlk}
For any nontrivial variety and any infinite cardinal $\ka$ there is a monomorphism 
$\thet_\ka\colon\LL_\ka \rightarrowtail \LL_\om\!^\ka$
 from $\LL_\ka$ into the $\ka$-th direct power of\/ $\LL_\om$.
\end{lemma}
\begin{proof}
That $\LL_\ka$  is embeddable in \emph{some} direct power of $\LL_\om$ is a consequence of the universal-algebraic analysis of varieties. The variety is generated by its $\LL_\om$, i.e.\ is equal to $\mathbf{HSP}\{\LL_\om\}$ \cite[4.132]{mcke:alge87}.
So all of its free algebras belong to $\mathbf{SP}\{\LL_\om\}$ \cite[4.119]{mcke:alge87}, making each $\LL_\ka$ isomorphic to a subalgebra of a direct power of $\LL_\om$.
An exercise in \cite[p.~77]{burr:univ81} states that the exponent for  this direct power can be taken to be $|\LL_\om|^\ka$. Here the size  $|\LL_\om|$ of $\LL_\om$ depends on the size of its signature,  but in any case $|\LL_\om|^\ka$  is at least $2^\ka$. 

To prove the result with the smaller exponent $\ka$, we use  the universal mapping property that defines free algebras.
Suppose that $\LL_\ka$ and $\LL_\om$ have generating sets $G_\ka$ and $G_\om$, of sizes $\ka$ and $\om$ respectively.
First we show that for each finite subset $i$ of $G_\ka$ there is a homomorphism 
$\thet_{i}\colon\LL_\ka \to \LL_\om$ that maps $i$ into $G_\om$ and is injective on the subalgebra of $\LL_\ka$ generated by $i$.
For, if $i$ is of size $n<\om$ we choose a subset $j$ of $G_\om$ that is of size $n$ and disjoint from $i$, and use the
freeness of $\LL_\ka$ to obtain a homomorphism $\thet_{i}\colon\LL_\ka\to\LL_\om$ that maps $i$ bijectively onto $j$.
If $a$ and $b$ belong to the subalgebra of $\LL_\ka$ generated by $i$, then  there are terms $s,t$ and a tuple 
$\vec{c}$ of distinct members of $i$ such that $a=s^{\LL_\ka}(\vec{c}\,)$ and  $b=t^{\LL_\ka}(\vec{c}\,)$.
Since homomorphisms commute with term functions, 
$\thet_{i}(a)=s^{\LL_\om}(\thet_i(\vec{c}\,))$ and $\thet_{i}(b)=t^{\LL_\om}(\thet_i(\vec{c}\,))$, where $\thet_i(\vec{c}\,)$ is the tuple got by applying $\thet_i$ to the coordinates of $\vec{c}$.
As $\LL_\om$ is free there is a homomorphism $\thet':\LL_\om\to\LL_\ka$ with $\thet'(\thet_i(\vec{c}\,))=\vec{c}$.
Then $\thet'(\thet_{i}(a))= \thet'(s^{\LL_\om}(\thet_i(\vec{c}\,)))
= s^{\LL_\ka}(\vec{c}\,)=a$ and likewise $\thet'(\thet_{i}(b))=b$. Thus if $a\ne b$, then $\thet_{i}(a)\ne \thet_{i}(b)$
as required to show that $\thet_{i}$ is injective on the subalgebra generated by $i$.

Now take $I$ to be the set of all finite  subsets  of $G_\ka$, and let $\thet\colon\LL_\ka \to \LL_\om\!^I$ be the homomorphism given by the product map
$a\mapsto\langle \theta_{i}(a):i\in I\rangle$. If $a\ne b$ in $\LL_\ka$, there is an $i\in I$ such that $a$ and $b$ belong to the subalgebra of $\LL_\ka$ generated by $i$. Then $\thet_{i}(a)\ne \thet_{i}(b)$ as above, which is enough to ensure that
$\thet(a)\ne \thet(b)$. Hence $\thet$ is a monomorphism.
But as $\ka$ is infinite,  $I$ is of cardinality $\ka$, so there is an isomorphism from $\LL_\om\!^I$ onto $ \LL_\om\!^\ka$ which composes with $\thet$ to give the desired $\thet_\ka\colon\LL_\ka \rightarrowtail \LL_\om\!^\ka$.
\end{proof}

Note that if $\ka$ is finite, then the above proof delivers a monomorphism $\LL_\ka \rightarrowtail \LL_\om$. For in that case $G_\ka$ is finite, so we can put $i=G_\ka$ and get that $\thet_{i}\colon\LL_\ka\to\LL_\om$ is injective on $\LL_\ka$ itself.

\begin{theorem} \label{equivalences}
Suppose $\ka_\Om$ is infinite, and let
 $\V$ be any variety from $\Om$ for which  the canonical extension of 
$\thet_{\ka_\Om}\colon\LL_{\ka_\Om} \rightarrowtail \LL_\om\!^{\ka_\Om}$ is a homomorphism.
Then the following are equivalent.
\begin{enumerate}[\rm(1)]
\item 
$\V$ is canonical.
\item
$\V$ contains the canonical extension $(\LL_\om\!^U)^\sg$ of the ultrapower $\LL_\om\!^U$ of $\LL_\om$ for every ultrafilter $U$ on any set.
\item
$\V$ contains the canonical extension $(\LL_\om\!^I)^\sg$ of the direct power $\LL_\om\!^I$ of $\LL_\om$ for every set $I$.
\item
$\V$ contains  $(\LL_\om\!^{\ka_\Om})^\sg$.
\end{enumerate}
\end{theorem}
\begin{proof}
(1) implies (2): $\V$ is closed under ultrapowers, so contains $\LL_\om\!^U$ for any ultrafilter $U$. Hence (2) is immediate from (1).

(2) implies (3): For any set $I$, by the direct power case of Theorem \ref{boolrepsig}, there is an isomorphism
\begin{equation*}\textstyle
\big(\LL_\om\!^I\big)^\sg \cong \prod_{U\in \beta I}\big(\LL_\om\!^U\big)^\sg.
\end{equation*}
Thus if (2) holds, then since $\V$ is closed under direct products and isomorphism, we get  $(\LL_\om\!^I)^\sg$ in $\V$.

(3) implies (4): Immediate.

(4) implies (1):       $\thet_{\ka_\Om}$ exists by Lemma \ref{injhomlk}, as $\ka_\Om$ is infinite.
By assumption,  $\thet_{\ka_\Om}\!^\sg\colon\LL_{\ka_\Om}\!^\sg \to \big(\LL_\om\!^{\ka_\Om}\big)^\sg$
is a homomorphism. It is also injective, since $\thet_{\ka_\Om}$ is injective, by the second part of Lemma \ref{presfsg}.
Thus $\LL_{\ka_\Om}\!^\sg$ is isomorphic to a subalgebra of $\big(\LL_\om\!^{\ka_\Om}\big)^\sg$, so if (4) holds then
$\LL_{\ka_\Om}\!^\sg$ belongs to $\V$, hence $\V$  is canonical by Theorem \ref{whencanonical}\eqref{whencanonical1}.
\end{proof}

Note that (1) directly implies (4)  in this theorem, and so the equivalence of (1) and (4) itself depends only on Theorem \ref{whencanonical}\eqref{whencanonical1}, hence does not require the detour via (2) and (3).  Thus it does not depend on the Boolean product analysis behind Theorem \ref{boolrepsig}.

As to the assumption that $\thet_{\ka_\Om}\!^\sg$ is a homomorphism, it is known that canonical extension
 does not preserve the property of being a homomorphism in general \cite[Example 3.8]{gehr:boun04} (although it does preserve  \emph{epi}morphisms as we have seen). But one context in which $f^\sg$ is guaranteed to be a homomorphism whenever $f$ is a homomorphism concerns \emph{monotone} algebras, those whose additional operations are,  in each coordinate,  either isotone or antitone. The class of all monotone  lattice-based algebras of a given type forms a category under homomorphisms on which canonical extension acts functorially, in particular taking homomorphisms to homomorphisms \cite[Theorem 5.4]{gehr:boun01}. Thus the equivalence of (1)--(4) in Theorem \ref{equivalences} holds for any variety of monotone lattice-ordered algebras. It is noteworthy  that the Theorem also holds if the assumption that
 $\thet_{\ka_\Om}\!^\sg$ is a homomorphism is weakened to just requiring that there is \emph{some} monomorphism from
 $\LL_{\ka_\Om}\!^\sg$ to $\big(\LL_\om\!^{\ka_\Om}\big)^\sg$.

\section{Polarities}

The following construction of complete lattices was given by Birkhoff in the first edition of his book 
\cite[Section 32]{birk:latt40}.
A \emph{polarity} is a structure $P=(X,Y,R)$ comprising sets $X$ and $Y$ and a binary relation 
$R\sub X\times Y$. 
This induces functions $\rho\colon\wp X\to\wp Y$ and $\lam\colon\wp Y\to\wp X$, where $\wp$ denotes powerset. Each
 subset $A$ of $X$ has the `right set'  $\rho A=\{y\in Y: \forall x\in A,xRy\}$. Each $B\sub Y$ has the `left set'
 $\lam B=\{x\in X: \forall y\in B,xRy\}$. The  functions $\rho$ and $\lam$ are inclusion-reversing and satisfy $A\sub\lam\rho A$ and $B\sub\rho\lam B$, i.e.\ they are a Galois connection between the posets $(\wp X,\sub)$ and $(\wp Y,\sub)$.
 A set $A\sub X$ is \emph{stable} if $\lam\rho A\sub A$ and hence $\lam\rho A=A$. 
 Since in fact $\lam\rho\lam B=\lam B$, the stable subsets of $X$ are precisely the sets $\lam B$ for all $B\sub Y$.
 The set $P^+$ of all stable subsets of $X$ is a complete lattice under the inclusion order, in which the meet $\meet G$ of any $G\sub P^+$ is its intersection 
 $\bigcap G$ and the join $\join G$ is $\lam\rho(\bigcup G)$. We call $P^+$ the \emph{stable set lattice} of $P$. Its greatest element is $X$ and its least element is $\lam\rho\emptyset=\lam Y$.
 We may write the Galois connection functions as $\rho_R,\lam_R$ when necessary to indicate which relation produces them.
 
The stable set lattice construction was used by Gehrke and Harding \cite{gehr:boun01} to obtain a canonical extension of any lattice $\LL$  as the stable set lattice of the polarity for which $X$ is the set of filters of $\LL$, $Y$ is the set of ideals, and $xRy$ iff $x\cap y\ne\emptyset$. The embedding $e$ in this case has $e(a)=\{x\in X:a\in x\}$.
 
 A fact that will be useful later is
 
 \begin{lemma} \label{meetjoin}
For any polarity $P$, if $A$ is a stable subset of $X$, then in $P^+$,
$A=\bigcap_{y\in\rho A}\lam\{y\}=\join_{x\in A}\lam\rho\{x\}.$       \hfill \qed

\end{lemma}

 An arbitrary polarity $P$ is a structure for the first-order language of the signature $\sL=\{\ov X,\ov Y,\ov R\}$ in which the relation symbols $\ov X$ and $\ov Y$ are unary and interpreted as the sets $X$ and $Y$ of $P$, while $\ov R$ is binary and interpreted as the relation $R$ of $P$. We use a set $\{v_n:n<\omega\}$ of individual variables ranging over $X\cup Y$. Thus $P$ is a model of the $\sL$-sentences
 $\forall v_0(\ov X(v_0)\lor\ov Y(v_0))$ and $\forall v_0\forall v_1(v_0\ov Rv_1\to\ov X(v_0)\land\ov Y(v_1))$.
 
 We will need to expand $\sL$ by adding various unary relation symbols $S$, typically interpreted as a subset of $X$. Then we define $\rho S(v_1)$ to be the formula
  $\forall v_0(S(v_0)\to v_0\ov Rv_1)$, and let $\lam\rho S(v_2)$ be $\forall v_1(\rho S(v_1)\to v_2\ov Rv_1)$. If $\sL'=\sL\cup\{S\}$ and a polarity $P$ is expanded to an $\sL'$-structure $P'$ by interpreting $S$ as the set $A\sub X$, then the formula $\rho S$ defines $\rho_RA$ in $P'$, i.e.\ 
 $P'\models (\rho S)[y]$ iff $y\in\rho_RA$. Hence $\lam\rho S$ defines $\lam_R\rho_RA$.
 Thus if $\mathsf{stable}$-$S$ is the sentence $\forall v_2(\lam\rho S(v_2)\to S(v_2))$, then $\mathsf{stable}$-$S$ expresses stability of $A$, i.e.\ $P'\models\mathsf{stable}$-$S$ iff $A$ is stable.
 
Next we discuss ultraproducts of polarities.
 Let $\{P_i=(X_i,Y_i,R_i):i\in I\}$ be a set of polarities and $U$ an ultrafilter on the index set $I$.
 The ultraproduct $\prod_U P_i$ is defined to be the polarity $(\prod_U X_i,\prod_U Y_i, R^U)$, where the binary relation 
 $R^U$ has $f^UR^Ug^U$ iff $\{i\in I:f(i)R_ig(i)\}\in U$.
When all the factors $P_i$ are equal to a single polarity $P=(X,Y,R)$, then the ultraproduct is the \emph{ultrapower} 
$P^U=(X^U,Y^U,R^U)$ of $P$ with respect to $U$. 
 
 Repeated use will be made of  \L o\'s's theorem, the so-called fundamental theorem of ultraproducts \cite[4.1.9]{chan:mode73}.
 If $\ph(v_0,\dots,v_n)$ is an $\sL$-formula and $f_0,\dots,f_n\in (\prod_I X_i)\cup(\prod_I Y_i)$, let
 \[
 \br{\ph(f_0,\dots,f_n)}=\{i\in I:P_i\models \ph[f_0(i),\dots,f_n(i)]\}.  
 \]
\L o\'s's Theorem states that  $\prod_U P_i\models\ph[f_0^U,\dots,f_n^U]$ iff $ \br{\ph(f_0,\dots,f_n)}\in U$.
In particular, if $\ph$ is a sentence, then $\prod_U P_i\models\ph$ iff $\{i:P_i\models\ph\}\in U$.

This result continues to hold for expansions of $\sL$ by unary relation symbols $S$. For instance, take a function $\al\in\prod_I\wp X_i$, and expand each $P_i$ to an $\sL\cup\{S\}$-structure $P_i'$ by interpreting $S$ as the set $\al(i)\sub X_i$.
Then the ultraproduct  $\prod_U P_i'$ is $\prod_U P_i$ with $S$ interpreted as the set
\begin{equation}\textstyle  \label{interpS}
\{f^U\in \prod_U X_i: \{i\in I:f(i)\in\al(i)\}\in U\}.
\end{equation}
\L o\'s's Theorem holds under this construction for $\sL\cup\{S\}$-formulas, and more generally for $\sL'$-formulas where 
$\sL'$ is any expansion of $\sL$ got by adding any number of unary relation symbols interpreted by functions $\al$ as above.
For such an $\sL'$ we reformulate \L o\'s's theorem as a result about definable sets that will be convenient for our purposes.
If $\al\in\prod_I\wp X_i$ and $f\in \prod_I X_i$, let $\br{f\in\al}=\{i\in I:f(i)\in\al(i)\}$. Then $\al^U\in\prod_U\wp X_i$ and we define
\begin{equation}  \textstyle  \label{defthet}
\thet(\al^U)=\{f^U\in \prod_U X_i: \br{f\in\al}\in U\},
\end{equation}
which is the set \eqref{interpS}. $\thet(\al^U)$ is a well-defined function of $\al^U$, since the set \eqref{interpS} is unchanged if $\al$ is replaced by any $\al'\in\prod_I\wp X_i$ with $\al^U=\al'\,^U$.

\begin{lemma} \label{Los}
Let $\ph(v_0,\dots,v_n)$ be any $\sL'$-formula. Suppose $\al\in\prod_I\wp X_i$ and $f_1,\dots,f_n \in (\prod_I X_i)\cup(\prod_I Y_i)$. If for all $i\in I$,
\[
\al(i)=\{x\in X_i:P'_i\models\ph[x,f_1(i),\dots,f_n(i)]\},
\]
then \enspace
$   \textstyle
\thet(\al^U)= \{f^U\in\prod_UX_i : \prod_U P_i'\models\ph[f^U,f_1^U,\dots,f_n^U]\}.
$
\end{lemma}
\begin{proof}
We have  $\br{f\in\al}=\{i\in I:P'_i\models\ph[f(i),f_1(i),\dots,f_n(i)]\}$,
so by \L o\'s's Theorem,
$\br{f\in\al}\in U$ iff  $\prod_U P_i'\models\ph[f^U,f_1^U,\dots,f_n^U]$.
\end{proof}
The case $n=0$ of this lemma states that if a formula $\ph(v_0)$ defines $\al(i)$ in $P_i'$ for all $i\in I$, then it defines 
$\thet(\al^U)$ in $\prod_U P_i'$.
 We apply this to establish the following relationship between the stable set lattice of $\prod_U P_i$ and the stable set lattices of the factors $P_i$.
  
\begin{theorem}   \label{ultembed}
There is a lattice monomorphism
$ \prod_U(P_i^+)   \mono    ( \prod_UP_i)^+ $ from the ultraproduct of the stable set lattices $P_i^+$ into the stable set lattice of the ultraproduct $\prod_U P_i$.
\end{theorem}
\begin{proof}
An element of $\prod_U(P_i^+) $ has the form $\al^U$ with $\al\in\prod_I(P_i^+)$. If we expand each $P_i$ to  $P_i'$ by interpreting a new symbol $S$ as the set $\al(i)$, then by \eqref{interpS} and \eqref{defthet} the interpretation of $S$ in $\prod_U P_i'$ is $\thet(\al^U)$. But the sentence $\mathsf{stable}$-$S$ is true in every factor $P_i'$, since $\al(i)\in P_i^+$, so by \L o\'s's Theorem it is true in $\prod_U P_i'$, implying that $\thet(\al^U)$ is stable in $\prod_U P_i$. Thus $\theta$ maps $\prod_U(P_i^+)$ into $ ( \prod_UP_i)^+$.
 
That $\thet$ is a monomorphism was shown in \cite{gold:meta74} in the case of modal algebras of subsets of Kripke frames (and extended to BAO's in general as complex algebras of relational structures in \cite{gold:vari89}). The proof that $\thet$ acting on $\prod_U(P_i^+)$  is well-defined, injective, and preserves binary meet (= intersection in stable set lattices) is just as  in \cite[Section 1.7]{gold:math93}.
To see that $\theta$ preserves the lattice bounds, note first that the greatest element of $\prod_U(P_i^+)$ is $1^U$ where $1$ is the greatest element of $\prod_I(P_i^+)$, having $1(i)=X_i$. Thus $1(i)$ is defined in $P_i$ by the formula $\ov X(v_0)$ for all $i\in I$, so by Lemma \ref{Los} this formula defines  $\thet(1^U)$ in $\prod_UP_i$. Hence  $\thet(1^U)=\prod_UX_i$, the greatest element of $( \prod_UP_i)^+ $.
The  least element of $\prod_U(P_i^+)$ is $0^U$ where $0$ is the least element of $\prod_I(P_i^+)$, having $0(i)=\lam_{R_i}Y_i$, which is defined in $P_i$ by the formula $\forall v_1(\ov Y(v_1)\to v_0\ov R v_1)$. So this formula defines $\thet(0^U)$ by Lemma \ref{Los}, making $\thet(0^U)=\lam_{R^U}\prod_U Y_i$,  the least element of $( \prod_UP_i)^+ $.

It remains  to show that $\thet$ preserves binary joins:
$
\thet(\al_1^U)\lor\thet(\al_2^U) = \thet(\al_1^U\lor\al_2^U).
$
Note first that
 $\al_1^U\lor\al_2^U$ is $(\al_1\lor\al_2)^U$, where $\al_1\lor\al_2$ is the point-wise join of $\al_1$ and $\al_2$ in 
$\prod_I(P_i^+)$, i.e.\ $(\al_1\lor\al_2)(i)= \al_1(i)\lor\al_2(i) \in P_i^+$.
Now take new unary symbols $S_1$ and $S_2$ and expand each $P_i$ to  $P_i'$ by interpreting $S_1$ as $\al_1(i)$ and $S_2$ as $\al_2(i)$. Let $\ph(v_0)$ be the formula
\[
\forall v_1\big[\forall v_2(S_1(v_2)\lor S_2(v_2)\to v_2\ov R v_1)\to v_0\ov Rv_1\big],
\]
expressing `$v_0\in\lam\rho(S_1\cup S_2)$'.
In each $P_i'$, $\ph(v_0)$ defines  $\lam_{R_i}\rho_{R_i}(\al_1(i)\cup\al_2(i))= (\al_1\lor\al_2)(i)$,
so by Lemma \ref{Los} it defines $\thet((\al_1\lor\al_2)^U)= \thet(\al_1^U\lor\al_2^U)$ in $\prod_U P_i'$.
But since $S_1$ defines $\thet(\al_1^U)$ and $S_2$ defines  $\thet(\al_2^U)$, the formula defines
$\thet(\al_1^U)\lor\thet(\al_2^U)$ in $\prod_U P_i'$, proving the desired preservation of the join.
\end{proof}

\section{MacNeille completions and  ultrapowers}

A \emph{MacNeille completion} of a lattice $\LL$ is a completion $e\colon \LL\mono\ov\LL$ of $\LL$ such that $e[\LL]$ is both meet-dense and join-dense in the complete lattice $\ov\LL$, i.e.\  every member of $\ov\LL$ is both a meet  of  elements of $e[\LL]$ and a join of  elements of $e[\LL]$. Every lattice has a MacNeille completion, and any two such completions are isomorphic by a unique isomorphism commuting with the embeddings of $\LL$ (see e.g.\ \cite{dave:intr90}).

We now show that any ultrapower  $(P^+)^U$ of the stable set lattice of a polarity $P=(X,Y,R)$ has a meet-dense and join-dense embedding into the stable set lattice $(P^U)^+$ of the ultrapower $P^U=(X^U,Y^U,R^U)$ of $P$.

\begin{theorem}  \label{Macppu}
For any polarity $P$ and any ultrafilter $U$ on a set $I$, the embedding 
$\thet\colon (P^+)^U\mono (P^U)^+$ is a MacNeille completion of $ (P^+)^U$.
\end{theorem}
\begin{proof}
$(P^U)^+$ is a complete lattice, and
 $\thet$ is the ultrapower case of the embedding provided by Theorem \ref{ultembed}, with
$\thet(\al^U)=\{f^U\in X^U: \br{f\in\al}\in U\}$ for any $\al\in (P^+)^I$.  Let $\image\thet$ be the $\thet$-image of
$(P^+)^U$ in $(P^U)^+$.

Now by Lemma \ref{meetjoin}, for any stable set $A\sub X^U$ we have
\[    
A=\bigcap\{\lam_{R^U}\{g^U\}: g^U\in\rho_{R^U} A\}=\join\{\lam_{R^U}\rho_{R^U}\{h^U\} : h^U\in A \},
\]
so to prove that $\image\thet$ is meet and join dense in $(P^U)^+$ it suffices to show that $\image\thet$ contains all sets of the form $\lam_{R^U}\{g^U\}$ and $\lam_{R^U}\rho_{R^U}\{h^U\}$.

First, for any $g\in Y^I$, define $\al\in(P^+)^I$ by putting $\al(i)=\lam_{R}\{g(i)\}$.
Let $\ph(v_0,v_1)$ be the formula $v_0\ov Rv_1$. Then $\al(i)=\{x\in X:P\models\ph[x,g(i)]\}$ for all $i\in I$,
so by Lemma \ref{Los}, $\thet(\al^U)=\{f^U\in X^U: P^U\models\ph[f^U,g^U]\}=\lam_{R^U}\{g^U\}$. 
Hence $\lam_{R^U}\{g^U\}\in\image\thet$.

Next, take any $h\in X^I$ and define $\al\in(P^+)^I$ by putting $\al(i)=\lam_{R}\rho_R\{h(i)\}$.
Let $\ph(v_0,v_1)$ be the formula $\forall v_2(v_1\ov Rv_2\to v_0\ov Rv_2)$.
Then for all $i\in I$, $\al(i)=\{x\in X:P\models\ph[x,h(i)]\}$, so  
$\thet(\al^U)=\{f^U\in X^U: P^U\models\ph[f^U,h^U]\}=\lam_{R^U}\rho_{R^U}\{h^U\}$.
\end{proof}

In \cite{gehr:macn06} Gehrke, Harding and Venema showed that any lattice $\LL$ has an extension $\LL^*$ such that the canonical extension $\LL^\sg$ of $\LL$ is embeddable into any MacNeille completion $\ov{\LL^*}$ of $\LL^*$.
(In fact this is shown for any monotone lattice expansion, a point we will return to later.)

 The embedding $\eta\colon \LL^\sg\mono\ov{\LL^*}$ is defined from the embedding $\varepsilon$ of $\LL$ into $\ov{\LL^*}$ by putting
\[
\eta x =\join\big\{\meet \{\varepsilon a:p\le a\in\LL\}:x\ge p\in K(\LL^\sg)\}
\]
(cf.\ \eqref{canextf}). Then $\eta$ is a complete lattice embedding of $ \LL^\sg$ into $\ov{\LL^*}$ provided that $\LL^*$ has a saturation property related to the size of $\LL$. An extension having this property can be obtained as an ultrapower of $\LL$, using the theory of saturation of ultrapowers \cite[6.1]{chan:mode73}. So for any lattice $\LL$ there exists an ultrafilter $U$ and an embedding of $ \LL^\sg$ into $\ov{\LL^U}$.

\begin{theorem} \label{PMac}
For any polarity $P$ there exists an ultrafilter $U$ and an embedding of the canonical extension $(P^+)^\sg$ into the stable set lattice $(P^U)^+$ of the $U$-ultrapower of $P$.
\end{theorem}
\begin{proof}
Putting  $\LL=P^+$ in the above analysis from \cite{gehr:macn06} gives $\LL^*= (P^+)^U$ for a suitable ultrafilter $U$. We take $\varepsilon\colon P^+\mono (P^U)^+$ to be the composition of the standard elementary embedding $\iota: P^+\mono(P^+)^U$ with the MacNeille completion
$\thet\colon (P^+)^U\mono (P^U)^+$ from Theorem \ref{Macppu}. Here $\iota$ maps each $A\in P^+$ to $\al_A{}^U$, where $\al_A$ is the constant map with value $A$; so $\varepsilon$ maps $A$ to $\{f^U\in X^U: \{i:f(i)\in A\} \}\in U$, which is the `enlargement' of $A$ in $X^U$ in the sense of nonstandard analysis.

By \cite[Theorem 32.2]{gehr:macn06},  $\varepsilon$ lifts to an embedding
$\eta\colon (P^+)^\sg\mono (P^U)^+$ as described above.
\end{proof}

 \section{Generating canonical varieties}
 
 A class $\cS$ of polarities has an associated class of lattices: the stable set lattices of all members of $\cS$. We are going to  show that if $\cS$ is closed under ultraproducts, then the variety generated by its class of stable set lattices is closed under canonical extensions.
But first we will  axiomatise the argument for this, so that it can be applied not just to polarities,  but to a class of structures of any kind for which ultraproducts are defined.

The symbols $\mono$ and $\epi$ will now be used to denote binary relations between algebras, writing $\A\mono\B$ to mean that there exists an injective homomorphism from $\A$ to $\B$, and $\A\epi\B$ to mean that there exists an surjective one. These relations are transitive on any similarity class of algebras. A variety $\V$ is closed under the relations in the sense that if $\A\mono\B\in\V$ then $\A\in\V$; and if $\A\epi\B$ and $\A\in\V$, then $\B\in\V$.

We envisage a situation involving the following four ingredients:
\begin{itemize}
\item 
A class $\Sigma$ of structures, of some type, that is closed under ultraproducts.
\item
A variety $\CC$ of algebras of some given algebraic signature.
\item
An operation $(-)^\sg\colon\CC\to\CC$ assigning to each algebra $\A\in\CC$ another algebra $\A^\sg\in\CC$.
\item
An operation $(-)^+\colon\Sigma\to\CC$ assigning to each structure $P\in\Sigma$ an algebra $P^+\in\CC$.
\end{itemize}
These ingredients will be said to form a \emph{canonicity framework} if they satisfy the following axioms
for all $\A,\B\in\CC$, all indexed subsets $\{P_i:i\in I\}$ of $\Sigma$,  and all $P\in\Sigma$.

\begin{enumerate}[({A}1)]
\item 
If  $\A\mono\B$ then  $\A^\sg\mono\B^\sg$, and if $\A\epi\B$ then $\A^\sg\epi\B^\sg$.
\item
$ \prod_U(P_i^+)   \mono    ( \prod_UP_i)^+ $, for any ultrafilter $U$ on $I$.
\item
There exists an ultrafilter $U$ such that $(P^+)^\sg\mono (P^U)^+$.
\item
$\big(\prod_I (P_i^+)\big)^\sg \mono \underset{U\in \beta I}{\prod} \big(\prod_U (P_i^+)\big)^\sg$.
\end{enumerate}

 \begin{theorem} \label{canframe}
For any canonicity framework, if $\cS$ is any subclass of $\Sigma$ that is closed under ultraproducts, then
the variety of algebras generated by  $\cS^+=\{P^+:P\in\cS\}$ is closed under the  operation     $(-)^\sg$.
\end{theorem}
\begin{proof}
The variety $\V$ in question is $\mathbf{HSP}\cS^+$. Thus if $\A$ is in $\V$, then there exists an algebra $\B$ and a set $\{P_i:i\in I\}\sub \cS$ such that $\B\epi \A$ and $\B\mono\C$ where $\C=\prod_I (P_i^+)$.
As $\CC$ is a variety including $\cS^+$, we have $\V\sub\CC$ and in particular $\A,\B$ and $\C$ are in $\CC$.
Hence $\A^\sg,\B^\sg$ and $\C^\sg$ are defined.
 By (A1), 
$\B^\sg\epi \A^\sg$ and $\B^\sg\mono\C^\sg$. So it suffices to prove that $\C^\sg$ is in $\V$ in order to obtain the desired conclusion that $\A^\sg$ is in the variety $\V$.

Let $U$ be any ultrafilter on $I$. Then $\prod_U (P_i^+)\in\CC$ as $\CC$ is closed under ultraproducts.
We will show that $\big(\prod_U (P_i^+)\big)^\sg$ belongs to $\V$. Put $P=\prod_U P_i$. Then $P\in\cS$ as $\cS$ is closed under ultraproducts, and $ \prod_U(P_i^+)   \mono  P^+$ by (A2). Hence
$\big( \prod_U(P_i^+)\big)^\sg   \mono  (P^+)^\sg$ by (A1).
By (A3) there is an ultrapower $P'$ of $P$ such that $(P^+)^\sg\mono (P')^+$. Hence
$\big( \prod_U(P_i^+)\big)^\sg   \mono  (P')^+$. But $ P'\in\cS$, as $\cS$ is closed under ultrapowers, so 
$(P')^+\in\cS^+\sub\V$. Hence  $\big( \prod_U(P_i^+)\big)^\sg\in\V$ as claimed.

Since $\V$ is closed under products,   it now follows that 
$\prod_{U\in \beta I}\big(\prod_U (P_i^+)\big)^\sg$ belongs to $\V$.
Hence by (A4),  $\C^\sg=\big(\prod_I (P_i^+)\big)^\sg$ is in $\V$, completing the proof.
\end{proof}

\begin{corollary}  \label{polcanon}
If $\cS$ is any class of polarities that is closed under ultraproducts, then the variety of lattices generated by $\cS^+$ is closed under canonical extensions.
\end{corollary}
\begin{proof}
We obtain a canonicity framework by taking $\Sigma$ as the class of all polarities, $\CC$ as the class of all bounded lattices, $(-)^\sg$ as the operation of canonical extension, and $(-)^+$ as the stable set lattice operation. (A1) then holds by Lemma \ref{presfsg}, (A2) by Theorem \ref{ultembed}, (A3) by Theorem \ref{PMac}, and (A4) as a consequence of Theorem \ref{boolrepsig} with $\LL_i=P_i^+$.
\end{proof}
As explained in Section \ref{sec1}, this result traces back to a theorem of Fine \cite{fine:conn75} that was stated for a modal logic characterised by an \emph{elementary} (i.e.\ first-order definable) class of Kripke frames. In fact closure of the class under ultraproducts suffices for the proof, but this weaker hypothesis does not make the result apply to more logics, or to more varieties in our present context. A variety generated by the algebras associated to an ultraproducts-closed class is also generated likewise by an elementary class:

\begin{theorem}
In a canonicity framework, for any ultraproducts-closed class $\cS$ the canonical variety $\V=\mathbf{HSP}\cS^+$ is equal to 
$\mathbf{HSP}\cS_\mathrm{el}^{\ +}$, where $\cS_\mathrm{el}$ is the smallest elementary class containing $\cS$.
\end{theorem}
\begin{proof}
Since $\cS\sub\cS_\mathrm{el}$, it suffices to show that $\cS_\mathrm{el}^{\ +}\sub\V$ to conclude that 
$\V=\mathbf{HSP}\cS_\mathrm{el}^{\ +}$. So let $P\in\cS_\mathrm{el}$.
As $\cS$ is closed under ultraproducts, there is some $P_1\in\cS$ such that $P\equiv P_1$, i.e.\ $P$ and $P_1$ satisfy the same first-order sentences  \cite[4.1.12]{chan:mode73}.
By the Keisler--Shelah Isomorphism Theorem \cite[6.1.15]{chan:mode73}, there is an ultrafilter $U$ such that $P^U\cong P_1^U$. Now   $(P_1^U)^{+}\in \cS^+$, as $\cS$ is closed under ultrapowers, and $(P^U)^{+}\cong (P_1^U)^{+}$, so 
$(P^U)^{+}\in\V$. But by (A2), $(P^+)^U\mono (P^U)^+$, so then $(P^+)^U\in\V$. As $P^+$ and its ultrapower $(P^+)^U$ satisfy the same equations, this implies $P^+\in\V$, showing that $\cS_\mathrm{el}^{\ +}\sub\V$ as required.
\end{proof}

Concerning the converse of Theorem \ref{canframe}, it has been shown \cite{gold:erdo04,gold:cano04} that there are many canonical varieties of BAO's that are not \emph{elementarily generated}, i.e.\ are not generated by the complex algebras of any elementary class of relational structures.  In \cite{gold:fine16} there is a structural analysis that illuminates the difference between canonical closure and elementary generation of a variety $\V$ of BAO's. To explain this briefly, recall from Section \ref{sec1} that any BAO  $\B$ has an associated canonical structure $\B_+$ whose complex algebra $(\B_+)^+$ is a canonical extension of $\B$. Now fix $\B$ to be the free algebra in $\V$ on $\om$-many generators. Then by the reasoning of Theorem \ref{equivalences} above, $\V$ is canonical iff
\begin{equation} \label{comp1}
\text{$((\B^U)_+)^+\in\V$  for all ultrafilters $U$.  }
\end{equation}
On the other hand it can be shown that $\V$ is elementarily generated iff
\begin{equation}   \label{comp2}
\text{$((\B_+)^U)^+\in\V$ for all ultrafilters $U$  }
\end{equation}
(\cite[11.5]{gold:math93} and\cite[3.4]{gold:fine16} give model-theoretic versions of this).
Comparing \eqref{comp1} and \eqref{comp2}, we see that a distinction between canonicity and elementary generation must be reflected in a  failure of the formation of canonical structures to commute with ultrapowers. The structures $(\B_+)^U$ and $(\B^U)_+$ cannot always be isomorphic. They need not have any property that would force $((\B_+)^U)^+$ to belong to $\V$ whenever $((\B^U)_+)^+$ does. In \cite{gold:fine16} an example is given in which $(\B_+)^U$ has size $2^\om$ while $(\B^U)_+$ has size $2^{2^\om}$.

A natural next step for the line of work of the present paper would be to apply Theorem \ref{canframe} to lattices with additional operations, or even to poset-based algebras. There is not yet a  theory of operations on the stable set  lattices of polarities that is as general as the construction by J\'onsson and Tarski \cite{jons:bool51} of $n$-ary operations on Boolean set algebras from $n+1$-ary relations. But there has been extensive work on the expansion of a polarity $P$ by ternary relations, as subsets of $X\times X\times Y$ and $X\times Y\times Y$, that are used to define residuated binary operations on $P^+$  which  model  various connectives in substructural logics \cite{dunn:cano05,gehr:gene06,cher:gene12,coum:rela14}.  Also in \cite{conr:cate16,conr:algo16} there are expansions of $P$ by binary relations, on $X$ and on $Y$ and from $X$ to $Y$ and $Y$ to $X$, that are used to model various unary modalities. In these investigations the operations induced on stable set lattices by the additional relations are first-order definable over the expanded polarity structure, and our framework can be adapted to that setting.
We leave that adaptation to another article.

In conclusion it is worth emphasising that almost everything we have done with canonical extensions in this paper has followed from their abstract description as dense and compact completions, without any reference to the particular nature of their elements. This is in keeping with an approach that J\'onsson favoured, and a reflection of the perspective that he expressed in another conference abstract \cite{jons:role92} as follows:
\begin{quote}\em
Many fundamental concepts of algebra can be formulated in a very general setting, and important results that were originally proved for special classes of algebras are actually true under quite weak assumptions. In fact, such results and their proofs often appear more natural when stripped of irrelevant assumptions.
\end{quote}



\bibliographystyle{spmpsci}


\end{document}